\documentclass[12pt]{article}
\usepackage{fancyhdr}
\usepackage{amsmath}
\usepackage{amsthm}
\usepackage{amssymb,amsfonts}
\usepackage{mathrsfs}
\usepackage{graphics}

\newtheorem{theorem}{Theorem}
\newtheorem{lemma}{Lemma}
\newtheorem{proposition}{Proposition}
\newtheorem{corollary}{Corollary}

\def\lam{\lambda}
\def\real{{\mathbb{R}}}
\def\R{{\real}}

\newcommand{\bel}{\begin{eqnarray}\label}
\newcommand{\eel}{\end{eqnarray}}
\newcommand{\bes}{\begin{eqnarray*}}
\newcommand{\ees}{\end{eqnarray*}}

\def\benu{\begin{enumerate}}
\def\eenu{\end{enumerate}}

\def\argmin{\mathop{\rm arg\, min}}
\def\real{{\mathbb{R}}}
\def\R{{\real}}
\def\E{{\mathbb{E}}}
\def\P{{\mathbb{P}}}

\def\complex{\mathop{{\rm I}\kern-.58em\hbox{\rm C}}\nolimits}

\def\rank{\hbox{\rm rank}}

\def\supp{\hbox{\rm supp}}

\def\mathbold{\boldsymbol} %\def\mathbold{\mathbf}
\def\ba{\mathbold{a}}

\def\bA{\mathbold{A}}

\def\scrB{{\mathscr B}}

\def\calD{{\cal D}}

\def\bfe{\mathbold{e}}

\def\bG{\mathbold{G}}
\def\tbG{{\widetilde{\bG}}}

\def\bI{\mathbold{I}}
\def\calI{{\cal I}}

\def\calL{{\cal L}}

\def\bP{\mathbold{P}}
\def\calP{{\cal P}}

\def\calQ{{\cal Q}}

\def\calR{{\cal R}}

\def\bT{\mathbold{T}}\def\hbT{{\widehat{\bT}}}
\def\calT{{\cal T}}

\def\bu{\mathbold{u}}

\def\bU{\mathbold{U}}
\def\scrU{{\mathscr U}}

\def\bv{\mathbold{v}}

\def\bV{\mathbold{V}}

\def\bw{\mathbold{w}}

\def\bW{\mathbold{W}}\def\tbW{{\widetilde{\bW}}}

\def\bX{\mathbold{X}}

\def\bY{\mathbold{Y}}

\def\bdelta{\mathbold{\delta}}

\def\bDelta{\mathbold{\Delta}}

\def\eps{\epsilon}

\def\lam{\lambda}

\def\kotimes{\otimes\cdots\otimes }
\def\ktimes{\times\cdots\times }

\def\plog{{\rm polylog}}

\begin{document}

\title{Incoherent Tensor Norms and Their Applications in Higher Order Tensor Completion}
\author{Ming Yuan$^\ast$ ~and~ Cun-Hui Zhang$^\dag$\\
University of Wisconsin-Madison ~and~ Rutgers University}

\footnotetext[1]{
Department of Statistics, University of Wisconsin-Madison, Madison, WI 53706. The research of Ming Yuan was supported in part by NSF FRG Grant DMS-1265202, and NIH Grant 1U54AI117924-01.}
\footnotetext[2]{
Department of Statistics and Biostatistics, Rutgers University, Piscataway, New Jersey 08854. The research of Cun-Hui Zhang was supported in part by NSF Grants DMS-1129626 and DMS-1209014}

\date{(\today)}
\maketitle

\begin{abstract}
In this paper, we investigate the sample size requirement for a general class of nuclear norm minimization methods for higher order tensor completion. We introduce a class of tensor norms by allowing for different levels of coherence, which allows us to leverage the incoherence of a tensor. In particular, we show that a $k$th order tensor of rank $r$ and dimension $d\ktimes d$ can be recovered perfectly from as few as $O((r^{(k-1)/2}d^{3/2}+r^{k-1}d)(\log(d))^2)$ uniformly sampled entries through an appropriate incoherent nuclear norm minimization. Our results demonstrate some key differences between completing a matrix and a higher order tensor: They not only point to potential room for improvement over the usual nuclear norm minimization but also highlight the importance of explicitly accounting for incoherence, when dealing with higher order tensors.
\end{abstract}

\newpage

\section{Introduction}
Data in the format of tensors, or multilinear arrays, arise naturally in many modern applications. A $k$th order hypercubic tensor of dimension $d\ktimes d$ has $d^k$ entries so that these datasets typically are of fairly large size even for moderate $d$ and small $k$. Therefore, it is oftentimes impractical to observe or store the entire tensor, which naturally brings about the question of tensor completion: How to reconstruct a $k$th order tensor $\bT\in \R^{d_1\ktimes d_k}$ from observations $\{\bT(\omega): \omega\in \Omega\}$ where $\Omega$ is a uniformly sampled subset from $[d_1]\ktimes [d_k]$? Here $[d]=\{1,\ldots, d\}$. The goal of this paper is to study in its full generality a class of tensor completion methods via nuclear norm minimization focusing on higher order tensors ($k\ge 3$).

\subsection{Tensor completion}
Obviously, for reconstructing $\bT$ from a subset of its entries to be possible at all, $\bT$ needs to have some sort of low dimensional structure which is often characterized by certain notion of low-rankness. In particular, let $\calL_j(\bX)$ be the linear subspace of $\R^{d_j}$ spanned by the mode-$j$ fibers:
$$\left\{\bX(a_1,\ldots,a_{j-1},\cdot,a_{j+1},\ldots,a_k)\in \R^{d_j}: a_1\in [d_1],\ldots,a_k\in [d_k]\right\}.$$
Denote by $r_j(\bX)$ the dimension of $\calL_j(\bX)$. The tuplet $\{r_1(\bX),\ldots, r_k(\bX)\}$ is the so-called Tucker ranks of $\bX$. It is not hard to see that there are a total of $O(r^{k-1}d)$ free parameters in specifying a $k$th order hypercubic tensor of dimension $d\ktimes d$ whose Tucker ranks are upper bounded by $r$, which suggests the possibility of recovering a large tensor of low rank from a fairly small fraction of the entries.

In addition to low-rankness, it is also essential to tensor completion that every entry of $\bT$ contains similar amount of information about the entire tensor so that missing any of them would not stop us from being able to reconstruct it -- a property that can be formally characterized through the {\it coherence} of the linear subspace $\calL_j(\bT)$. See, e.g, Cand\`es and Recht (2008). More specifically, the coherence of an $r$ dimensional linear subspace $U$ of $\R^d$ is defined as
$$
\mu(U)={d\over r}\max_{1\le i\le d}\|\bP_U\bfe_i\|^2_{\ell_2} 
= \frac{\max_{1\le i\le d}\|\bP_U\bfe_i\|_{\ell_2}^2}{d^{-1}\sum_{i=1}^d\|\bP_U\bfe_i\|_{\ell_2}^2}, 
$$
where $\bP_U$ is the orthogonal projection onto $U$ and  $\bfe_i$'s are the canonical basis for $\R^d$. We call a tensor $\bX$ $\mu_\ast$-incoherent if
$$
\mu_j(\bX):=\mu(\calL_j(\bX))\le \mu_\ast.
$$

An especially popular class of techniques to tensor completion is based on nuclear norm minimization where we seek among all tensors that agree with $\bT$ on all observed entries the one with the smallest nuclear norm.

\subsection{Nuclear norm minimization}
Recall that the spectral and nuclear norms of a tensor $\bX\in\R^{d_1\times\cdots\times d_k}$ are defined as
$$
\|\bX\|=\sup_{\bu_j\in \R^{d_k}: \|\bu_j\|_{\ell_2}\le 1}\langle \bX, \bu_1\kotimes \bu_k\rangle
$$
and
$$
\|\bX\|_\ast=\sup_{\bY\in\R^{d_1\times\cdots\times d_k}: \|\bY\|\le 1}\langle \bX, \bY\rangle,
$$
respectively, where $\langle\cdot,\cdot\rangle$ is the usual vectorized inner product, and $\|\cdot\|_{\ell_p}$ stands for the usual $\ell_p$ norm in a vector space. The usual nuclear norm minimization proceeds by solving the following convex optimization problem:
\bel{eq:nnm}
\min_{\bX\in \R^{d_1\ktimes d_k}} \|\bX\|_\ast\qquad{\rm subject\ to\ } \calP_\Omega\bX=\calP_\Omega\bT,
\eel
where $\calP_\Omega: \R^{d_1\times\cdots\times d_k}\to \R^{d_1\times\cdots\times d_k}$ is a linear operator such that
$$
\calP_\Omega \bX(\omega)=\left\{\begin{array}{ll}\bX(\omega)& {\rm if\ }\omega\in \Omega\\ 0& {\rm otherwise}\end{array}\right..
$$
The solution to (\ref{eq:nnm}) is our reconstruction of $\bT$. This approach was first introduced for matrices, that is $k=2$, by Cand\`es and Recht (2008) and Cand\`es and Tao (2009). Similar approaches have also been adopted later for higher order tensors. See, e.g., Liu et al. (2009), Signoretto, Lathauwer and Suykens (2010), Gandy et al. (2011), Tomioka, Hayashi and Kashima (2010), Tomioka et al. (2011), Mu et al. (2013), Jain and Oh (2014), and Yuan and Zhang (2014), among many others.

Of particular interest here is the requirement on the cardinality $|\Omega|$, which we shall refer to as the sample size, to ensure that $\bT$ can be reconstructed perfectly (with high probability) via nuclear norm minimization (\ref{eq:nnm}).   It is now well understood that in the case of matrices ($k=2$), a $d\times d$ incoherent matrix of rank $r$ can be recovered with high probability if $|\Omega|\gtrsim rd\cdot\plog (d)$ under suitable conditions, where $a\gtrsim b$ means that $a>Cb$ for some constant $C>0$ independent of $r$ and $d$, and $\plog(d)$ stands for a certain polynomial of $\log(d)$. See, e.g., Recht (2010), and Gross (2011) among many others. It is clear that this sample size requirement is nearly optimal since the number of free parameters needed to specify a $d\times d$ rank $r$ matrix is of the order $O(rd)$. 

The situation for higher order tensors is more complicated as there are multiple ways to generalize the matrix style nuclear norm. A common practice is to first reshape a high order tensor to a matrix and then apply the techniques such as (\ref{eq:nnm}) to the unfolded matrix. In doing so, one recasts the problem of completing a $k$th order tensor, say of dimension $d\ktimes d$, as a problem of completing a $d^{\lfloor k/2\rfloor}\times d^{\lceil k/2\rceil}$ matrix. Following the results  for matrices, it can be shown that the sample size requirement for recovering a $k$th order hypercubic tensor of dimension $d\ktimes d$ and whose Tucker ranks are bounded by $r$ in this fashion is
$$|\Omega|\gtrsim r^{\lfloor k/2\rfloor}d^{\lceil k/2\rceil}\plog(d).$$
However, as Yuan and Zhang (2014) recently pointed out, this strategy is often suboptimal and direct minimization of the tensor nuclear norm yields a tighter sample size requirement at least when $k=3$. In particular they show that, under suitable conditions, a $d\times d\times d$ tensor whose Tucker ranks are bounded by $r$ can be recovered perfectly with high probability if
$$|\Omega|\gtrsim (r^{1/2}d^{3/2}+r^2d)\plog(d).$$
Following their argument, it is also possible to show that, when $k>3$, the sample size required for exact recovery via tensor nuclear norm minimization is
$$|\Omega|\gtrsim d^{k/2}{\rm poly}(r,\log(d)),$$
where ${\rm poly}(\cdot,\cdot)$ is a certain polynomial in both arguments. However, it remains unknown to what extent such a sample size requirement is tight for nuclear norm minimization based approaches. The main goal of this paper is to address this question. Indeed, we show that this sample size condition for higher order tensor can be much improved.

\subsection{Incoherent nuclear norm minimization}
The key ingredient of our approach is to define a new class of tensor nuclear norms that explicitly account for the incoherence of the linear subspaces spanned by the fibers of a tensor in defining its nuclear norm. More specifically, for a $\bdelta=(\delta_{1},\ldots,\delta_{k})\in (0,1]^k$, let
\bes
\scrU_{j_1j_2}(\bdelta)
= \left\{\bu_1\otimes \cdots\otimes \bu_k: \|\bu_j\|_{\ell_2} \le 1,\ \forall j; \|\bu_j\|_{\ell_\infty}\le \delta_{j}, \forall j\neq j_1, j_2\right\}
\ees
be the set of all rank-one tensors satisfying incoherent conditions in ``directions'' other than $j_1$ and $j_2$.  Then
\bes
\scrU(\bdelta) = \bigcup_{1\le j_1<j_2\le k} \scrU_{j_1j_2}(\bdelta)
\ees
is the collection of all rank-one tensors satisfying certain incoherence conditions in all but two directions. For a $k$th order tensor $\bX\in\R^{d_1\times\cdots\times d_k}$, define a norm 
\bes
\|\bX\|_{\circ,\bdelta} = \sup_{\bY\in \scrU(\bdelta)}\langle \bY,\bX \rangle.
\ees
Note that when $\bdelta={\bf 1}:=(1,\ldots,1)^\top$, the $\ell_\infty$ constraint in defining $\|\bX\|_{\circ,\bdelta}$ becomes inactive so that $\|\bX\|_{\circ,{\bf 1}} =\|\bX\|$, the usual tensor spectral norm. We can view $\|\cdot\|_{\circ,\bdelta}$ as a {\it incoherent} spectral norm. We can also define the incoherence nuclear norm as the dual of the incoherence spectral norm:
\bes
\|\bX\|_{\star,\bdelta} =  \sup_{\|\bY\|_{\circ,\bdelta}\le 1}\langle \bY,\bX \rangle,
\ees
so that $\|\bX\|_{\star,{\bf 1}}$ reduces to the usual tensor nuclear norm. 

Instead of minimizing the usual tensor nuclear norm, we now consider recovering $\bT$ via the following nuclear norm minimization problem:
\bel{eq:innm}
\min_{\bX\in \R^{d_1\ktimes d_k}} \|\bX\|_{\star,\bdelta}\qquad{\rm subject\ to\ } \calP_\Omega\bX=\calP_\Omega\bT.
\eel
It is clear that (\ref{eq:innm}) reduces to the usual nuclear norm minimization (\ref{eq:nnm}) if $\bdelta={\mathbf 1}$. But as we shall see later, it could be extremely beneficial to take smaller values for $\delta_j$s. Our goal is to investigate the appropriate choices of $\bdelta$, and when $\bT$ can be recovered through the incoherent nuclear norm minimization (\ref{eq:innm}).

\subsection{Outline}
Our main result provides a sample size requirement for recovering an incoherent and low rank tensor $\bT\in \R^{d_1\ktimes d_k}$ via (\ref{eq:innm}). In particular, our result implies that a $k$th order hypercubic tensor of dimension $d\ktimes d$ whose Tucker ranks are bounded by $r$ can be reconstructed perfectly by the solution of (\ref{eq:innm}) with appropriate choices of $\bdelta$, as long as
$$|\Omega|\gtrsim (r^{(k-1)/2}d^{3/2}+r^{k-1}d)(\log(d))^2.$$
This represents a drastic improvement over the requirement for the usual nuclear norm minimization. It is especially worth noting that, perhaps somewhat surprisingly, the sample size given above depends on the order $k$ only through the rank $r$ which, in most situations of interest, is small. It is also instructive to look at the case when a tensor is of finite rank, that is $r=O(1)$. The sample size requirement in such cases becomes $O(d^{3/2}(\log(d))^2)$ for any fixed order $k$, which suggests the possibility of a tremendous amount of data reduction even for moderate $k$s.

In establishing the sample size requirement for the proposed incoherent nuclear norm minimization approach, we developed various algebraic properties of incoherent tensor norms including a characterization of the subdifferential of the incoherent tensor nuclear norm which generalizes earlier results for matrices (Watson, 1992) and for the usual nuclear norm with third order tensors (Yuan and Zhang, 2014). 

Also essential to our analysis are large deviation bounds under the incoherent spectral norm we derived for randomly sampled tensors, which may be of independent interest. These probabilistic bounds show a tighter concentration behavior of random tensors under incoherent norm than under the usual spectral norm, an observation we exploited to establish tighter sample size requirement for tensor completion. We note that concentration inequalities such as the ones presented here are the basis for many problems beyond tensor completion. For examples, it is plausible that these bounds could prove useful in developing improved sampling schemes for higher order tensor sparsification. See, e.g., Nguyen, Drineas and Tran (2015). These applications are beyond the scope of the current paper and we shall leave them for future studies.

The rest of the paper is organized as follows. In the next section, we introduce the notion of incoherent tensor norms and establish some algebraic properties of these norms useful for our analysis. In Section \ref{sec:large}, we derive large deviation bounds for randomly sampled tensors. Building on the tool developed in Sections \ref{sec:alg} and \ref{sec:large}, we provide the sample size requirement for the incoherent nuclear norm minimization in Section \ref{sec:main}. We conclude with some discussions and remarks in Section \ref{sec:dis}

\section{Subdifferential of Incoherent Tensor Nuclear Norm}
\label{sec:alg}
Note that the optimization problem (\ref{eq:innm}) is convex. In order to show that $\bT$ can be recovered via (\ref{eq:innm}), it suffices to find a member from the subdifferential of $\|\cdot\|_{\star,\bdelta}$ at $\bT$ that can certify it as the unique solution to (\ref{eq:innm}). To this end, we need to characterize the subdifferential of $\|\cdot\|_{\star,\bdelta}$, which we shall do in this section.

We first note several immediate yet useful observations of the incoherent spectral and nuclear norms. We shall make repeated use of these simple properties without mentioning in the rest of paper.
\begin{proposition}\label{tensor-prop-1}
For any tensor $\bX\in \R^{d_1\ktimes d_k}$ and $\bdelta\in (0,1]^k$, 
$$
\|\bX\|_{\rm HS}^2:=\langle \bX, \bX\rangle\le \|\bX\|_{\circ,\bdelta}\|\bX\|_{\star,\bdelta},
$$
and
$$
\|\bX\|_{\circ,\bdelta} \le\|\bX\|\le \|\bX\|_{\rm HS}\le\|\bX\|_{\ast}\le \|\bX\|_{\star,\bdelta}.
$$
\end{proposition}

Recall that, for a tensor $\bX$, $\calL_j(\bX)$ is the linear subspace of $\R^{d_j}$ spanned by the mode-$j$ fibers of $\bX$. Denote by $\bP_j(\bX)$ the orthogonal projection to $\calL_j(\bX)$. For brevity, we omit the dependence of $\bP_j$ and $\calL_j$ on $\bX$ hereafter when no confusion occurs. Write
$$
\calQ_{\bX}^0 = \bP_1\kotimes \bP_k.
$$
It is clear that for any $\bu_j\in \R^{d_j}$, we have
$$
\langle \bu_1\otimes \cdots\otimes \bu_k,\bX \rangle = \langle \bP_1\bu_1\otimes \cdots\otimes \bP_k\bu_k,\bX \rangle,
$$
This immediately implies that
\begin{proposition}\label{tensor-prop-2}
Let $\delta_j\ge \max_{\|\bu\|_{\ell_2}\le 1}\|\bP_j(\bX) \bu\|_{\ell_\infty}$, for $j=1,\ldots,k$. Then, for any tensor $\bW\in\R^{d_1\ktimes d_k}$, 
\bes
\|\calQ_{\bX}^0\bW\|_{\circ,\bdelta} = \|\calQ_{\bX}^0\bW\| \le \|\bW\|_{\circ,\bdelta}.
\ees
Consequently, $\|\bX\|_{\star,\bdelta}=\|\bX\|_{\ast}$.
\end{proposition}

Propositions \ref{tensor-prop-1} indicates that the incoherent nuclear norm is greater than the usual nuclear norm in general. But Proposition \ref{tensor-prop-2} shows that the two norms are equal if a tensor is indeed incoherent. This gives some intuition on the potential benefits of minimizing the incoherent instead of the usual nuclear norm. Because more penalty is levied on tensors that are not incoherent, compared with the usual nuclear norm minimization (\ref{eq:nnm}), it is more plausible that the solution of (\ref{eq:innm}) is incoherent. Given that the truth is known apriori to be incoherent, it is more likely that incoherent tensor nuclear norm minimization produces exact recovery. This advantage will be more precisely quantified by the much refined sample size requirement we shall establish later.

We are now in position to describe a characterization of the subdifferential of $\|\cdot\|_{\star,\bdelta}$. Let $\bP_j^\perp=\bI-\bP_j$ be the projection to the orthogonal complement $\calL_j^\perp$ of $\calL_j$ in $\R^{d_j}$. Write
$$
\calQ_{\bX} = \calQ_{\bX}^0+ \sum_{j=1}^k \bP_1\kotimes \bP_{j-1}\otimes \bP_j^\perp\otimes \bP_{j+1}\kotimes\bP_k.
$$
It is easy to see that
$$
\calQ_{\bX}^\perp:= \calI - \calQ_{\bX} = \sum_{1\le j_1<j_2 \le k}\calQ^\perp_{\bX,j_1,j_2},
$$
where $\calI$ is the identity operator on the appropriate space, and
$$
\calQ^\perp_{\bX,j_1,j_2} =\bP_1\kotimes\bP_{j_1-1}\otimes \bP_{j_1}^\perp\otimes \bP_{j_1+1}\kotimes \bP_{j_2-1}\otimes\bP_{j_2}^\perp\otimes \calI\kotimes \calI.
$$
We note that $\calQ^\perp_{j_1,j_2}$ is the orthogonal projection to the linear space of all $\bu_1\kotimes\bu_k$ such that $\bu_j$ is in either $\calL_j$ or $\calL_j^\perp$ for $1\le j\le j_2$ and that $j_1$ and $j_2$ are the only indices with $\bu_j\in \calL_j^\perp$. 

\begin{theorem}\label{tensor-prop-3} 
Let $\delta_j\ge \max_{\|\bu\|_{\ell_2}\le 1}\|\bP_j(\bX) \bu\|_{\ell_\infty}$, for $j=1,\ldots,k$. Then there exists an $\bW_0\in\R^{d_1\ktimes d_k}$ such that
$$\calQ_{\bX}^0\bW_0=\bW_0,\qquad \|\bW_0\|_{\circ,\bdelta}=1,\qquad  {\rm and}\qquad \|\bX\|_{\star,\bdelta}=\langle\bW_0,\bX\rangle.$$
Moreover, for any $\bY\in\R^{d_1\ktimes d_k}$
\bes
\|\bY\|_{\star,\bdelta} \ge \|\bX\|_{\star,\bdelta} + \frac{2}{k(k-1)}\|\calQ_{\bX}^\perp\bY\|_{\star,\bdelta}+\langle \bW_0,\bY-\bX\rangle. 
\ees
\end{theorem}

\vskip 15pt
\begin{proof}[Proof of Theorem \ref{tensor-prop-3}]
Let $\tbW_0$ be the dual of $\bX$ satisfying $\|\tbW_0\|_{\circ,\bdelta} =1$ and $\langle \tbW_0,\bX\rangle = \|\bX\|_{\star,\bdelta}$.  Set $\bW_0=\calQ_{\bX}^0\tbW_0$. Since $\bX = \calQ_{\bX}^0\bX$ and $\calQ_{\bX}^0$ is an orthogonal projection, we have $\calQ_{\bX}^0\bW_0=\bW_0$, $\calQ_{\bX}^\perp\bX=0$ and $\|\bX\|_{\star,\bdelta}=\langle \bW_0,\bX\rangle\le \|\bW_0\|_{\circ,\bdelta}\|\bX\|_{\star,\bdelta}$ . This, along with Proposition \ref{tensor-prop-2}, proves the first statement.

To prove the second statement, we first show that for any $\bW_1\in \R^{d_1\ktimes d_k}$ such that $\|\bW_1\|_{\circ, \bdelta}\le 2/\{k(k-1)\}$, we have
\bel{eq:sbd}
\|\bW_0+\calQ_{\bX}^\perp\bW_1\|_{\circ,\bdelta}\le 1.
\eel
To this end, note first that
\bes
\|\bW_0+\calQ_{\bX}^\perp \bW_1\|_{\circ,\bdelta}&=&\sup_{\bu_1\kotimes \bu_k\in \scrU(\bdelta)}\langle \bu_1\kotimes \bu_k,\bW_0+\calQ_{\bX}^\perp \bW_1\rangle\\
&=&\max_{1\le j_1<j_2\le k}\left\{\sup_{\bu_1\kotimes \bu_k\in \scrU_{j_1j_2}(\bdelta)}\langle \bu_1\kotimes \bu_k,\bW_0+\calQ_{\bX}^\perp \bW_1\rangle\right\}.
\ees
It then suffices to show that for any $1\le j_1<j_2\le k$, and $\bu_1\kotimes \bu_k\in \scrU_{j_1,j_2}(\bdelta)$,
$$
\langle \bu_1\kotimes \bu_k,\bW_0+\calQ_{\bX}^\perp \bW_1\rangle\le 1.
$$
As the statement is not specific to the index label, we assume without loss of generality that $j_1 = 1$ and $j_2 = 2$; Otherwise, a different decomposition of $\calQ_{\bX}^{\perp}$ is needed beginning with the projection $\calI\kotimes \calI\otimes \bP_{j_1}^\perp\otimes\calI\kotimes\calI\otimes \bP_{j_2}^\perp\otimes \calI\kotimes \calI$. Recall that
\bes
\langle \bu_1\kotimes \bu_k,\calQ_{\bT}^\perp\bW_1 \rangle
&\le& \sum_{1\le j_3<j_4\le k}\langle \bu_1\kotimes \bu_k,\calQ_{j_3,j_4}^\perp\bW_1 \rangle
\cr &\le& {1\over 2} k(k-1)\max_{1\le j_3<j_4\le k}\langle \bu_1\kotimes \bu_k,\calQ_{j_3,j_4}^\perp\bW_1 \rangle. 
\ees
By definition,
\bes
&&\langle \bu_1\kotimes \bu_k,\calQ_{j_3,j_4}^\perp\bW_1 \rangle\\
&=&\langle \bP_1\bu_1\kotimes\bP_{j_3-1}\bu_{j_3-1}\otimes \bP_{j_3}^\perp\bu_{j_3}\kotimes \bP_{j_4-1}\bu_{j_4-1}\otimes\bP_{j_4}^\perp\bu_{j_4}\kotimes \bu_k,\bW_1 \rangle.
\ees
Because $\|\bu\|_{\ell_\infty}\le \delta_{j}$ for all $j\ge 2$ and $\|\bP_j\bu\|_{\ell_\infty}\le \delta_{j}\le \delta_j$ for all $\bu\in \R^{d_k}$ with $\|\bu\|_{\ell_2}\le 1$, we have
\bes
&&\langle \bu_1\kotimes \bu_k,\calQ_{j_3,j_4}^\perp\bW_1 \rangle\\
&\le&\|\bP_{j_3}^\perp\bu_{j_3}\|_{\ell_2}\|\bP_{j_4}^\perp\bu_{j_4}\|_{\ell_2}\sup_{\bu_1\kotimes \bu_k\in \scrU_{j_3j_4}(\bdelta)}\langle \bu_1\kotimes \bu_k,\bW_1 \rangle\\
&\le&\|\bP_{j_3}^\perp\bu_{j_3}\|_{\ell_2}\|\bP_{j_4}^\perp\bu_{j_4}\|_{\ell_2}\|\bW_1\|_{\circ,\bdelta}\\
&\le&{2\over k(k-1)}\|\bP_{j_3}^\perp\bu_{j_3}\|_{\ell_2}\|\bP_{j_4}^\perp\bu_{j_4}\|_{\ell_2}.
\ees
Together with the fact that
\bes
\langle \bu_1\kotimes \bu_k,\calQ_{\bX}^0\bW_0\rangle&=&\langle \bP_1\bu_1\kotimes\bP_k\bu_k, \bW_0\rangle\\
&\le& \|\bW_0\|_{\circ,{\bf 1}}\prod_{j=1}^k \|\bP_j\bu_j\|_{\ell_2},
\ees
we get, for any $\bu_1\kotimes \bu_k\in \scrU_{j_1j_2}(\bdelta)$,
\bes
&&\langle \bu_1\kotimes \bu_k,\calQ_{\bX}^0\bW_0+\calQ_{\bX}^\perp \bW_1\rangle\\
&\le&\prod_{j=1}^k \|\bP_j\bu_j\|_{\ell_2}+\max_{1\le j_3<j_4\le k}\|\bP_{j_3}^\perp\bu_{j_3}\|_{\ell_2}\|\bP_{j_4}^\perp\bu_{j_4}\|_{\ell_2}\\
&\le&\max_{1\le j_3<j_4\le k}\left\{\|\bP_{j_3}\bu_{j_3}\|_{\ell_2}\|\bP_{j_4}\bu_{j_4}\|_{\ell_2}+\|\bP_{j_3}^\perp\bu_{j_3}\|_{\ell_2}\|\bP_{j_4}^\perp\bu_{j_4}\|_{\ell_2}\right\}\\
&\le&\max_{1\le j_3<j_4\le k}\left\{\left(\|\bP_{j_3}\bu_{j_3}\|_{\ell_2}^2+\|\bP_{j_3}^\perp\bu_{j_3}\|_{\ell_2}^2\right)^{1/2}\left(\|\bP_{j_4}\bu_{j_4}\|_{\ell_2}^2+\|\bP_{j_4}^\perp\bu_{j_4}\|_{\ell_2}^2\right)^{1/2}\right\}\\
&=&1.
\ees

It then follows that
\bes
\|\bY\|_{\star,\bdelta} - \|\bX\|_{\star,\bdelta} 
&\ge& \max_{\|\bW_1\|_{\circ,\bdelta}\le 2/\{k(k-1)\}}\langle \bW_0+\calQ_{\bT}^\perp\bW_1,\bY-\bX\rangle\\
&= &\frac{\|\calQ_{\bX}^\perp\bY\|_{\star,\bdelta}}{k(k-1)/2}+\langle \bW_0,\bY-\bX\rangle. 
\ees
This completes the proof.
\end{proof}
\vskip 15pt

Theorem \ref{tensor-prop-3} provides a sufficient condition for a tensor to be in the subdifferential $\partial \|\bX\|_{\star,\bdelta}$. More specifically, it states that there exists a $\bW_0$ so that for any $\bW_1$ such that $\bW_1=\calQ_{\bX}^\perp \bW_1$ and $\|\bW_1\|_{\circ,\bdelta}\le 2/\{k(k-1)\}$,
$$
\bW_0+\bW_1\in \partial \|\bX\|_{\star,\bdelta}.
$$
This characterization generalizes the earlier result by Yuan and Zhang (2014) for the special case when $k=3$ and $\bdelta={\mathbf 1}$.

\section{Concentration under Incoherent Spectral Norm}
\label{sec:large}

A main technical tool for many tensor related problems is the large deviation bounds for the spectral norm of a random tensor. We shall use such bounds, in particular, to construct a dual certificate for (\ref{eq:innm}) later on. 

Let $\bA\in \R^{d_1\ktimes d_k}$ be an arbitrary but fixed tensor. We are interested in the behavior of randomly sampled tensors
$$
\bX_i=(d_1\cdots d_k)\calP_{\omega_i}\bA, \qquad i=1,\ldots, n,
$$
where $\omega_i$s are iid uniform random variables on $[d_1]\ktimes[d_k]$. Write
$$
\bar{\bX}={1\over n}\left(\bX_1+\cdots+\bX_n\right).
$$
It is clear that $\E\bar{\bX}=\bA$. We are interested in bounding the incoherent spectral norm of its deviation from the mean $\|\bar{\bX}-\bA\|_{\circ,\bdelta}$.

Denote by
$$
\|\bA\|_{\max}=\max_{\omega\in [d_1]\ktimes [d_k]} |\bA(\omega)|.
$$
For brevity, write
$$
d={1\over k}\sum_{1\le j\le k} d_j,\qquad {\rm and}\qquad d_\ast=(d_1\cdots d_k)^{1/k},
$$
and
$$
\delta_\ast=(\delta_1\cdots \delta_k)^{1/k},\qquad{\rm and}\qquad \delta_{**}=\min_{1\le j_1<j_2\le k}\sqrt{\delta_{j_1}\delta_{j_2}}.
$$
We first give a general concentration bound.

\begin{theorem}\label{con-th-2}
Suppose that $d$ is sufficiently large such that
$$
{8e\over 9\log 2}k^2(\log d)^3\le d.
$$
For any $\alpha>0$ and 
\bes
t \ge 160(3\alpha+7)\frac{k}{n}\sqrt{d\log d_*}(2\delta_\ast d_\ast)^k\|\bA\|_{\max}
\max_{1\le j_1<j_2\le k}\left\{\left({n\over \delta_{j_1}^2d_{j_1}\delta_{j_2}^2d_{j_2}}
+\frac{\log d}{\delta_{j_1}^2\delta_{j_2}^2}\right)\right\}^{1/2},
\ees
then
\bes
&&\P\left\{\left\|\bar{\bX}-\bA\right\|_{\circ,\bdelta} \ge t \right\}\le {1\over 2}k^2d^{-\alpha}+{1\over 4(\log 2)^2}k^2(\log d)^2\times\\
&&\hskip 50pt\times\left\{\exp\left(-{9nt^2\over 64kd_\ast^k\|\bA\|_{\max}^2\log d_\ast}\right) + \exp\left(-{9nt\over 32k\delta_\ast^k\delta_{\ast\ast}^{-2}d_\ast^k\|\bA\|_{\max}\log d_\ast}\right)\right\}. 
\ees
\end{theorem}

\vskip 15pt
The proof relies on the following result which is an extension of Lemma 9 of Yuan and Zhang (2015) to accommodate an $\ell_\infty$ bound. 

\begin{lemma}\label{lm-A} Let $\delta \in [1/\sqrt{d},1]$ and $m$ be an integer with 
$2^{m/2} < \delta\sqrt{d}\le 2^{(m+1)/2}$. Then,
\bes
\max_{\|\bu\|_{\ell_2}\le 1,\|\bu\|_{\ell_\infty}\le\delta}\bu^\top\ba 
\le (2/c)\max\left\{\bw^\top\ba: \|\bw\|_{\ell_2}\le c,\bw\in\{\pm c 2^{j/2}/\sqrt{2d}, j=0,\ldots,m\}^d\right\}
\ees
for all $0<c\le 1$. Moreover, 
\bes
\left|\left\{\bw: \|\bw\|_{\ell_2}\le c,\bw\in\{\pm c2^{j/2}/\sqrt{2d}, j=0,\ldots,m\}^d\right\}\right|\le 
\exp\big(1.344+3.082 \times d\big).
\ees 
\end{lemma}

For brevity, the proof of Lemma \ref{lm-A} is deferred to the Appendix. We now present the proof of Theorem \ref{con-th-2}.

\vskip 15pt
\begin{proof}[Proof of Theorem \ref{con-th-2}]
The standard symmetrization argument gives 
\bes
\P\left\{\left\|\bar{\bX}-\bA\right\|_{\circ,\bdelta}\ge 3t\right\}
&\le& \max_{\bu_1\kotimes \bu_k\in \scrU(\bdelta)}
\P\left\{\left\langle \bar{\bX}-\bA,\bu_1\kotimes \bu_k\right\rangle\ge t\right\}\\
&&+4\,\P\left\{\left\|{1\over n}\sum_{i=1}^n\epsilon_i\bX_i\right\|_{\circ,\bdelta}\ge t\right\}.
\ees
See, e.g., Gin\'{e} and Zinn (1984). For any fixed $\bu_1\kotimes \bu_k\in \scrU(\bdelta)$, we have 
\bes
\E \left\langle \bX_i, \bu_1\kotimes \bu_k\right\rangle
&=&\left\langle \bA, \bu_1\kotimes \bu_k\right\rangle,\\
|\left\langle \bX_i, \bu_1\kotimes \bu_k\right\rangle|
&\le& (d_1\cdots d_k)\left(\|\bu_1\|_{\ell_\infty}\cdots \|\bu_k\|_{\ell_\infty}\right)\|\bA\|_{\max}\\
&\le& (d_1\cdots d_k)(\delta_1\cdots \delta_k)
\|\bA\|_{\max}/\delta_{**}^2,
\ees
and
$$
{\rm var}(\left\langle \bX_i, \bu_1\kotimes \bu_k\right\rangle)
\le \E \left\langle \bX_i, \bu_1\kotimes \bu_k\right\rangle^2 \le (d_1\cdots d_k) \|\bA\|_{\max}^2.
$$
Therefore, by the Bernstein inequality,
\bes
\P\left\{\left\|\bar{\bX}-\bA\right\|_{\circ,\bdelta}\ge 3t\right\}
&\le&\exp\left(-{nt^2\over 4d_\ast^k \|\bA\|_{\max}^2}\right)
+\exp\left(-{(3/4)\delta_{**}^2nt \over {d_\ast^k}\delta_\ast^k\|\bA\|_{\max}}\right)\\
&&+4\P\left\{\left\|{1\over n}\sum_{i=1}^n\epsilon_i\bX_i\right\|_{\circ,\bdelta}\ge t\right\}
\ees
We now proceed to bound the last term on the right hand side.

For brevity, write $\bY_i=\epsilon_i\bX_i$ and
$$
\bar{\bY}={1\over n}\sum_{i=1}^n\epsilon_i\bX_i.
$$
Recall that
$$
\|\bar{\bY}\|_{\circ,\bdelta}=\max_{1\le j_1<j_2\le k}\max_{\bu_1\kotimes \bu_k\in \scrU_{j_1j_2}(\bdelta)}\langle \bar{\bY},\bu_1\kotimes \bu_k\rangle.
$$
Hence,
$$
\P\left\{\|\bar{\bY}\|_{\circ,\bdelta}\ge t\right\}\le \sum_{1\le j_1<j_2\le k}\P\left\{\max_{\bu_1\kotimes \bu_k\in \scrU_{j_1j_2}(\bdelta)}\langle \bar{\bY},\bu_1\kotimes \bu_k\rangle\ge t\right\}.
$$
We now bound each of the summands on the right hand side. To fix ideas, we shall treat only the case when $j_1=1$ and $j_2=2$ without loss of generality.

It follows from Lemma \ref{lm-A} that 
\bes
\max_{\bu_1\kotimes\bu_k\in\scrU_{1,2}(\bdelta)}\left\langle\bar{\bY},\bu_1\kotimes\bu_k\right\rangle 
\le  2^{k+1}\max_{\bu_1\kotimes\bu_k\in\scrU_{1,2}^\ast(\bdelta)}\langle \bar{\bY}, 
\bu_1\otimes \cdots\otimes \bu_k\rangle.
\ees
where
$$
\scrU^\ast_{1,2}(\bdelta) = \left\{\bu_1\otimes \cdots\otimes \bu_k \in \scrU_{1,2}(\bdelta): \|\bu_j\|_{\ell_2}\le c_j, 
\bu_j\in \{\pm 2^{j/2}c_j/\sqrt{2d_j}, j=0,\ldots,m_j\}^{d_j}\right\}
$$
with $m_j=\lceil\log_2(d_j)-1\rceil$ for $j=1,2$, and 
$m_j=\lceil\log_2(\delta_j^2d_j)-1\rceil$ for $j>2$. 
We choose 
 $1/\sqrt{2}\le c_j\le 1$ such that $ \{\pm 2^{j/2}c_j/\sqrt{2d_j}, j=0,\ldots,m_j\} 
= \{\pm 2^{-j/2}, j=2,\ldots,m_j+2\}$ for $j=1,2$, and $c_j=1$ for $j>2$. 
As $d_1+\cdots+d_k=kd$ and $d\ge 2$, 
\bes
|\scrU^\ast_{1,2}(\bdelta)| \le \exp\big(4kd\big). 
\ees
For $\bU=\bu_1\kotimes\bu_k\in \scrU^\ast_{1,2}(\bdelta)$, define 
\bes
A_m &=& \{(a_1,a_2): |\bu_1(a_1)\bu_2(a_2)| = 2^{-m/2}\}, 
\cr B_m &=& \{(a_3,\ldots,a_k): (a_1,a_2)\in A_m, 
(a_1,\ldots,a_k)\in\Omega\}, 
\ees
and% matricize $\bU$ as $\bU = \bU_{1,2}\otimes \bU_{3,...,k}$ with 
\bes
\bU_{1,2}=\bu_1\otimes \bu_2,\ \bU_{3,...,k}=\bu_3\kotimes \bu_k. 
\ees
Here and in the sequel, we omit the dependence of $\{A_m,B_m,\bU_{1,2},\bU_{3,...,k}\}$ 
on $\bU$ and $B_m$ on $\Omega$ when no confusion occurs. 
For $\bU\in\scrU_{1,2}^\ast(\bdelta)$ and any integer $m_{1,2}\ge 0$, 
\bes
\langle \bar{\bY}, \bU\rangle 
= \langle \bar{\bY}, (\calP_{C_{1,2}}\bU_{1,2})\otimes \bU_{3,...,k}\rangle
+ \sum_{4\le m\le m_{1,2}} \langle \bar{\bY}, (\calP_{A_m}\bU_{1,2})\otimes (\calP_{B_m}\bU_{3,...,k})\rangle,
\ees
where
$$C_{1,2} = \{(a_1,a_2): |\bU_{1,2}(a_1,a_2)| \le 2^{-m_{1,2}/2-1/2}\}.$$ 
We note that $A_m=\emptyset$ for $m\le 3$. 

Write
$$
\nu_{1,2}(\bar{\bY})=\max_{a_{1}\in [d_1], a_2\in [d_2]}\left|\left\{(a_1,\ldots, a_k)\in \supp(\bar{\bY}): a_j\in [d_j], j\ge 3\right\}\right|.
$$
%By the Chernoff bound, it is not hard to see that
We argue that
\bel{eq:chern}
\P\left\{\nu_{1,2}(\bar{\bY})\le (3\alpha+7)\left({n\over d_{1}d_{2}}+\log d\right)\right\}\le d^{-\alpha}.
\eel
When $n/(d_1d_2)\ge \log d$, we can apply Chernoff bound to get, for any fixed $a_1\in[d_1]$ and $a_2\in[d_2]$
\bes
&&\P\left\{\left|\left\{(a_1,\ldots, a_k)\in \supp(\bar{\bY}): a_j\in [d_j], j\ge 3\right\}\right|\ge (3\alpha+7){n\over d_{1}d_{2}}\right\}\\
&\le& \exp[-(\alpha+2)n/(d_1d_2)]\le d^{-(\alpha+2)}.
\ees
Similarly, when $n/(d_1d_2)<\log d$, we can also apply Chernoff bound to get
$$
\P\left\{\left|\left\{(a_1,\ldots, a_k)\in \supp(\bar{\bY}): a_j\in [d_j], j\ge 3\right\}\right|\ge (3\alpha+7)\log d\right\}\le d^{-(\alpha+2)}.
$$
Equation (\ref{eq:chern}) then follows from an application of the union bound.

We shall now proceed conditional on the event that
$$\nu_{1,2}(\bar{\bY})\le \nu_\ast:=(3\alpha+7)\left({n\over d_{1}d_{2}}+\log d\right).$$
Under this event,
\bes
|B_m| \le \nu_\ast|A_m|.
\ees

Observe that for any $\bU=\bu_1\kotimes \bu_k\in \scrU_{1,2}^\ast(\bdelta)$, 
$$
|A_m|\le 2^{m},\ \|\bU_{3,...,k}\|_{\max} = \|\bu_3\kotimes \bu_k\|_{\max} \le \delta_{3,...,k}. 
$$
with $\delta_{3,...,k}=\delta_3\cdots\delta_k$. 
For integers $0\le \ell \le m\le m_{1,2}$ define, 
\bes
\scrB_{1,2}(m,\ell) 
&=& \Big\{\bV = (\calP_{A_m}\bU_{1,2})\otimes (\calP_B\bU_{3,...,k}): |A_m|\le 2^{m-\ell}, 
\cr &&\qquad |B|\le \nu_\ast|A_m|, \bU_{1,2}\otimes \bU_{3,...,k}\in \scrU_{1,2}^\ast(\bdelta)\Big\}. 
\ees
It follows that for $\bU\in\scrU_{1,2}^\ast(\bdelta)$ and integers $a_m\ge 0$ with $2^{m-a_m-1}\le |A_m|\le 2^{m-a_m}$, 
\bes
(\calP_{A_m}\bU_{1,2})\otimes (\calP_{B_m}\bU_{3,...,k})\in \scrB_{1,2}(m,\ell),\quad a_m\le \ell. 
\ees
As
$$\sum_{m=4}^{m_{1,2}} 2^{-(a_m\wedge(m-3))} 
\le 1+2\sum_{m=4}^{m_{1,2}} |A_m|/2^m\le 1+2\|\bU_{1,2}\|_{\rm F}^2\le 3$$
for all $\bU\in\scrU_{1,2}^\ast(\bdelta)$, 
\bes
&& \sum_{4\le m\le m_{1,2}} \langle \bar{\bY}, (\calP_{A_m}\bU_{1,2})\otimes (\calP_{B_m}\bU_{3,...,k})\rangle
\cr &\le& \sum_{4\le m\le m_{1,2}} 2^{-(a_m\wedge(m-3))/2-\ell_m/2}
\max_{\bV\in \scrB_{1,2}(m,a_m\wedge(m-3))}2^{(a_m\wedge(m-3))/2+\ell_m/2}\langle \bar{\bY}, \bV\rangle
\cr &\le&\left(3\sum_{m=4}^{m_{1,2}}2^{-\ell_m}\right)^{1/2}
\max_{4\le m\le m_{1,2}}\max_{0\le\ell\le m-3}
\max_{\bV\in \scrB_{1,2}(m,\ell)}2^{\ell/2+\ell_m/2}\langle \bar{\bY}, \bV\rangle
\ees
for any nonnegative integers $\ell_m$. Here $a\wedge b=\min\{a,b\}$. It follows that if
$$\left(3\sum_{m=4}^{m_{1,2}}2^{-\ell_m}\right)^{1/2}\le 4,$$
then
\bel{pf-1}
\langle \bar{\bY}, \bU\rangle 
&\le& \max_{\bU\in\scrU_{1,2}^\ast(\bdelta)}
\langle \bar{\bY}, (\calP_{C_{1,2}}\bU_{1,2})\otimes \bU_{3,...,k}\rangle 
\cr && + 4\max_{4\le m\le m_{1,2}}\max_{0\le\ell\le m-3}
\max_{\bV\in \scrB_{1,2}(m,\ell)}2^{\ell/2+\ell_m/2}\langle \bar{\bY}, \bV\rangle. 
\eel
We note that $\calP_{C_{1,2}}=\calI$ when $m_{1,2}\le 3$. 

We have $|\scrU_{1,2}^\ast(\bdelta)|\le e^{4kd}$. 
To bound the cardinality of $\scrB_{1,2}(m,\ell)$, we pick 
\bes
m_{1,2} = \max\left\{\lfloor \log_2(4d/(\nu_*\log d_*))\rfloor, 0 \right\}, 
\ees
so that
$$\nu_* 2^{m_{1,2}}\log d_* \le 4d \le \nu_* 2^{m_{1,2}+1}\log d_*$$
if $\nu_*\log d_*\le 4d$ and $m_{1,2}=0$ otherwise. 
Moreover, for $4\le m\le m_{1,2}$, we pick integers $\ell_m$ satisfying 
\bes
\max\left\{2^{m-m_{1,2}},\frac{9}{8k\log d_*}\right\} \le 2^{-\ell_m} < \max\left\{2^{m-m_{1,2}},\frac{9}{4k\log d_*}\right\}. 
\ees
As
$$m_{1,2}\le \log_2d\le k\log(d_*)/\log 2,$$
we have 
\bes
\left(3\sum_{m=4}^{m_{1,2}}2^{-\ell_m}\right)^{1/2} 
\le \left(\frac{27(1+m_{1,2}-3)}{4k\log d_*}\right)^{1/2} 
\le \left(\frac{27}{4\log 2}\right)^{1/2} \le 3.121. 
\ees
We note that $\bU_{1,2}$ takes value $\pm 2^{-m/2}$ on $A_m$ 
and $\bU_{3,\ldots,k}$ takes value in $\pm 2^{j/2}/(\prod_{j=3}^k \sqrt{2d_j})$ for $j=0,\ldots,m_3+\ldots+m_k$.  
Let $m_{**}=k\log_2(\delta_*^2d_*)$. 
As $m_j=\lceil\log_2(\delta_j^2d_j)-1\rceil$ for $j>2$, each element of $\bU_{3,\ldots,k}$ 
has at most $2m_{**}+2$ possible values. It follows that 
\bes
\log |\scrB_{1,2}(m,\ell)| 
&\le& \log\left(\sum_{j=1}^{2^{m-\ell}}{d_1d_2\choose j}{d_3\cdots d_k\choose \lfloor\nu_* j\rfloor}
2^j(2m_{**}+2)^{\lfloor\nu_* j\rfloor}\right)
\cr &\le& \nu_* 2^{m-\ell}\left\{\log\left(\frac{ed_3\ldots d_k}{\nu_* 2^{m-\ell}}\right) + \log(2m_{**}+2) \right\}
\cr && + 2^{m-\ell}\left\{\log\left(\frac{ed_1d_2}{2^{m-\ell}}\right) + \log 2 \right\} + \log 2. 
\ees
As $x\log(y/x^2)$ is increasing in $x$ for $0< x \le \sqrt{y}/e$ and $4\le m\le m_{1,2}-\ell_m$, 
\bes
&& 2^{-(m-\ell)/2}\log |\scrB_{1,2}(m,\ell)| 
\cr &\le& \nu_* 2^{(m_{1,2}-\ell_m)/2}\left\{\log\left(\frac{ed_3\ldots d_k}{\nu_* 2^{m_{1,2}-\ell_m}}\right) + \log(2m_{**}+2) \right\}
\cr && + 2^{(m_{1,2}-\ell_m)/2}\left\{\log\left(\frac{ed_1d_2}{2^{m_{1,2}-\ell_m}}\right) + 2\log 2\right\}
\cr &\le& \nu_* 2^{(m_{1,2}-\ell_m)/2}\left\{\log\left(\frac{e(d_1d_2)^{1/\nu_*}d_3\ldots d_k}{\nu_* 2^{m_{1,2}-\ell_m}}\right) + \log(2m_{**}+2) \right\}
\cr &\le& \nu_* 2^{-\ell_m/2}\left({4d\over \nu_*\log d_*}\right)^{1/2}
\log\left(\frac{d_*^k e (d_1d_2)^{1/\nu_*}2^{\ell_m}(2m_{**}+2)}{d_1d_24d/\log d_*}\right).
%\cr&\le & 2^{-\ell_m/2}k\sqrt{4\nu_*d\log d_*}. 
\ees
%The last inequality above is a consequence of 
Note that
\bes
&& e (d_1d_2)^{1/\nu_*} 2^{\ell_m}(2m_{**}+2)\log d_* 
\cr &\le& (d_1d_2)^{1/\{(1+\alpha)\log d\}}(8e/9)k(\log d_*)^2\{2k\log_2(\delta_*^2d_*)+2\}
\cr &\le& 4d_1d_2d,
\ees  
where the last inequality follows from the fact that $d_\ast<d$ and the assumption that $d$ is sufficiently large. Thus,
$$
2^{-(m-\ell)/2}\log |\scrB_{1,2}(m,\ell)|\le 2^{-\ell_m/2}k\sqrt{4\nu_*d\log d_*}. 
$$
It follows that 
\bes
\log |\scrB_{1,2}(m,\ell)| \le 2^{(m-\ell-\ell_m)/2} k\sqrt{4\nu_*d\log d_*} \le 4kd ,\qquad \forall\ 0\le\ell\le m\le m_{1,2}. 
\ees

For any fixed $\bV\in \scrB_{1,2}(m,\ell)$, write
$Z_i=\left\langle \bY_i,\bV\right\rangle$. Then
$$
\left\langle \bar{\bY}, \bV\right\rangle={1\over n}(Z_1+\cdots+Z_n).
%\left\langle \bar{\bY}, \calP_{A}(\bv_1\otimes \cdots\otimes \bv_k)\right\rangle={1\over n}(Z_1+\cdots+Z_n).
$$
We have
$$\|\bV\|_{\max}\le 2^{-m/2}\delta_{3,...,k}\qquad {\rm and} \qquad \|\bV\|_{\rm HS}^2\le 2^{-\ell}.$$ 
Thus, as $\bY_i=\epsilon_i\bX_i$ and $\bX_i=(d_1\cdots d_k)\calP_{\omega_i}\bA$, 
we have
$$|Z_i|\le d_\ast^k\|\bA\|_{\max}\|\bV\|_{\max}\le 2^{-m/2}\delta_{3,...,k}d_\ast^k\|\bA\|_{\max}$$
and
$$
{\rm var}(Z_i)\le \E(Z_i^2)\le d_\ast^k\|\bA\|_{\max}^2\|\bV\|_{\rm HS}^2 \le 2^{-\ell}d_\ast^k\|\bA\|_{\max}^2.
$$
It follows from the Bernstein inequality and the union bound that 
\bes
&&\P\left\{\max_{\substack{\bV\in \scrB_{1,2}(m,\ell)}}\left\langle \bar{\bY},\bV\right\rangle
\ge 2^{-(\ell+\ell_m)/2}t\right\}\\
&\le& |\scrB_{1,2}(m,\ell)|\exp\left(-{n2^{-\ell-\ell_m}t^2\over 2^{1-\ell}d_\ast^k\|\bA\|_{\max}^2+(2/3)2^{-m/2}\delta_{3,...,k}d_\ast^k\|\bA\|_{\max}2^{-(\ell+\ell_m)/2}t}\right)\\
&\le& \exp\left(4kd -{n2^{-\ell_m}t^2\over 4d_\ast^k\|\bA\|_{\max}^2}\right)
+ \exp\left(2^{(m-\ell-\ell_m)/2} k\sqrt{4\nu_*d\log d_*} 
-{(3/4)2^{(m-\ell-\ell_m)/2}nt \over \delta_{3,...,k}d_\ast^k\|\bA\|_{\max}}\right).%\\
%&\le& \exp\left(-4kd\right)
%+ \exp\left(-\sqrt{4k\nu_*d}\right),
\ees
The condition on $t$ implies that
$$t \ge {8\over 3n}(\delta_{3,...,k}d_\ast^k\|\bA\|_{\max}) k\sqrt{4\nu_*d\log d_*}.$$
Together with the fact that $2^{-\ell_m}\ge (9/8)/(k\log d_*)$, we get
\bes
\frac{n 2^{-\ell_m} t^2}{4d_\ast^k\|\bA\|_{\max}^2} 
\ge \frac{2\delta_{3,...,k}^2d_\ast^kk^2(4\nu_*d\log d_*)}{n k \log d_*} 
\ge (d_1d_2\nu_*/n) 8kd \ge 8kd. 
\ees
Therefore,
$$
\exp\left(4kd -{n2^{-\ell_m}t^2\over 4d_\ast^k\|\bA\|_{\max}^2}\right)\le \exp\left( -{n2^{-\ell_m}t^2\over 8d_\ast^k\|\bA\|_{\max}^2}\right)\le \exp\left( -{9nt^2\over 64kd_\ast^k\|\bA\|_{\max}^2\log d_\ast}\right).
$$
Similarly, we have
$$
{(3/4)2^{(m-\ell-\ell_m)/2}nt \over \delta_{3,...,k}d_\ast^k\|\bA\|_{\max}} \ge 2 \cdot 2^{(m-\ell-\ell_m)/2}k\sqrt{4\nu_*d\log d_*},
$$
which implies that
\bes
&&\exp\left(2^{(m-\ell-\ell_m)/2} k\sqrt{4\nu_*d\log d_*} 
-{(3/4)2^{(m-\ell-\ell_m)/2}nt \over \delta_{3,...,k}d_\ast^k\|\bA\|_{\max}}\right)\\
&\le&\exp\left(-{3\over 8}\cdot{2^{(m-\ell-\ell_m)/2}nt \over \delta_{3,...,k}d_\ast^k\|\bA\|_{\max}}\right)\\
&\le&\exp\left(-{9\over 32}\cdot{nt \over (k\log d_\ast)^{1/2}\delta_{3,...,k}d_\ast^k\|\bA\|_{\max}}\right).
\ees

For $m_{1,2}\ge 1$, we have
$$2^{-(m_{1,2}+1)/2}\le \sqrt{(\nu_*\log d_*)/(4d)},$$
so that 
\bes
{3\over 4}nt 2^{(m_{1,2}+1)/2}/(\delta_{3,...,k}d_\ast^k\|\bA\|_{\max}) 
\ge 2 k\sqrt{4\nu_*d\log d_*}\sqrt{4d/(\nu_*\log d_*)}=8kd. 
\ees
As
$$|\langle \eps_i\bX_i, (\calP_{C_{1,2}}\bU_{1,2})\otimes \bU_{3,...,k}\rangle| 
\le 2^{-(m_{1,2}+1)/2}\delta_{3,\ldots,k}d_*^k\|\bA\|_{\max},$$
we have  
\bes
&& \P\left\{\max_{\bU\in\scrU_{1,2}^\ast(\bdelta)}
\langle \bar{\bY}, (\calP_{C_{1,2}}\bU_{1,2})\otimes \bU_{3,...,k}\rangle \ge t \right\}
\cr &\le & |\scrU^\ast_{1,2}(\bdelta)|\max_{\bU\in\scrU_{1,2}^\ast(\bdelta)}
\P\left\{\langle \bar{\bY}, (\calP_{C_{1,2}}\bU_{1,2})\otimes \bU_{3,...,k}\rangle \ge t \right\}
\cr &\le & \exp\left(4kd -{nt^2\over 2d_\ast^k\|\bA\|_{\max}^2
+2^{1-(m_{1,2}+1)/2}\delta_{3,\ldots,k}d_*^k\|\bA\|_{\max}t/3}\right)\\
&\le&\exp\left(-{nt^2\over 4d_\ast^k\|\bA\|_{\max}^2}\right)+\exp\left(-{3d^{1/2}nt\over 2\delta_{3,\ldots,k}d_*^k\|\bA\|_{\max}(\nu_*\log d_*)^{1/2}}\right).
%\cr &\le & \exp\left(4kd -{nt^2\over nt^2/(16kd)+nt^2/(16kd)}\right)
\ees
Finally, for $m_{1,2}=0$, we have $\nu_\ast> 4d/\log d_*$, so that 
the condition on $t$ still implies 
\bes
{(3/4)nt \over \delta_{3,...,k}d_\ast^k\|\bA\|_{\max}}
\ge 2k\sqrt{4\nu_*d\log d_*} \ge 8kd. 
\ees

Putting the above probability bounds together via (\ref{pf-1}), we find that 
\bes
&& \P\left\{\max_{\bU\in \scrU_{1,2}(\bdelta)}\langle \bar{\bY}, \bU\rangle 
\ge 2^{k+1}5 t \right\}
\cr &\le&\P\left\{\max_{\bU\in \scrU^\ast_{1,2}(\bdelta)}\langle \bar{\bY}, \bU\rangle \ge 5t \right\}
\cr &\le& \Big(1+2+\ldots+(m_{1,2}-2)\Big)\times\\
&&\times\left\{\exp\left( -{9nt^2\over 64kd_\ast^k\|\bA\|_{\max}^2\log d_\ast}\right) + \exp\left(-{9\over 32}\cdot{nt \over (k\log d_\ast)^{1/2}\delta_{3,...,k}d_\ast^k\|\bA\|_{\max}}\right)\right\}. 
\ees
As $m_{1,2}\le \log_2 d$, the proof is then completed in the light of (\ref{eq:chern}).
\end{proof}
\vskip 15pt

It is instructive to examine the case of hypercubic tensors where $d_1=\cdots=d_k=d$ and we take $\delta_1=\cdots=\delta_k=\delta_\ast$. The following is an immediate consequence of Theorem \ref{con-th-2}.
\begin{corollary}
Let $\bA\in \R^{d\ktimes d}$ be a $k$th order tensor, and $\delta_1=\cdots=\delta_k=\delta\in (0,1]$, then there exists constant $c_1,c_2>0$ depending on $k$ only such that, for any $\beta>0$,
\bel{eq:sample}
\left\|\bar{\bX}-\bA\right\|_{\circ,\bdelta}\le c_1(1+\beta)
\max\left\{\left({\log d\over n}\right)^{1/2}\delta^{k-2}d^{k-1/2}, 
\left({\log d\over n}\right)\delta^{k-2}d^{k+1/2}\right\}\|\bA\|_{\max},
\eel
with probability at least $1-c_2d^{-\beta}$. 
\end{corollary}
\medskip

Note that the second term on the right hand side of (\ref{eq:sample}) decreases with $\delta$, indicating a tighter concentration bound for $\bar{\bX}-\bA$ when it dominates the first term. The bound (\ref{eq:sample}) immediately suggests an effective sampling scheme to approximate incoherent tensors in terms of the usual spectral norm. Suppose that $\bA$ is $\mu$-incoherent so that
$$
\max_{\|\bu\|_{\ell_2}\le 1}\|\bP_j(\bA)\bu\|_{\ell_\infty}\le \sqrt{\mu r_j(\bA)/d},\qquad j=1,\ldots, k.
$$
Then we can take $\delta=2\sqrt{\mu r/d}$ where $r=\max_j r_j(\bA)$. Equation (\ref{eq:sample}) now becomes
$$
\left\|\bar{\bX}-\bA\right\|_{\circ,\bdelta}\lesssim 
(\mu r)^{k/2-1}\max\left\{\left({\log d\over n}\right)^{1/2}d^{(k+1)/2}, 
\left({\log d\over n}\right)d^{(k+3)/2}\right\}\|\bA\|_{\max}.
$$
Let $\widehat{\bA}$ be the projection of $\bar{X}$ onto the space $\calT_\mu$ of $\mu$-incoherent tensors:
$$
\widehat{\bA}=\argmin_{\bY\in \calT_\mu} \|\bar{\bX}-\bY\|_{\circ,\bdelta}.
$$
By triangular inequality, $\|\widehat{\bA}-\bA\|_{\circ,\bdelta} \le 2\left\|\bar{\bX}-\bA\right\|_{\circ,\bdelta}$, so that  
$$ 
\|\widehat{\bA}-\bA\|_{\circ,\bdelta} \lesssim
(\mu r)^{k/2-1}\max\left\{\left({\log d\over n}\right)^{1/2}d^{(k+1)/2}, 
\left({\log d\over n}\right)d^{(k+3)/2}\right\}\|\bA\|_{\max}.
$$
Because both $\widehat{\bA}$ and $\bA$ are $\mu$-coherent. Their difference $\widehat{\bA}-\bA$ must be $\sqrt{2}\mu$-coherent. In the light of Proposition \ref{tensor-prop-2}, we know $\|\widehat{\bA}-\bA\|=\|\widehat{\bA}-\bA\|_{\circ,\bdelta}$, so that 
\bel{eq:appr}
\|\widehat{\bA}-\bA\| \lesssim
(\mu r)^{k/2-1}\max\left\{\left({\log d\over n}\right)^{1/2}d^{(k+1)/2}, 
\left({\log d\over n}\right)d^{(k+3)/2}\right\}\|\bA\|_{\max}.
\eel
In other words, we can approximate $\bA$ up to the same error bound given by (\ref{eq:sample}), but in terms of the usual spectral norm.

For illustration purposes, consider a more specific case when $\bA$ admits an orthogonal decomposition
$$
\bA=\sum_{i=1}^r \bu_1^{(i)}\kotimes \bu_k^{(i)},
$$
for some $\bu_j^{(i)}\in \R^{d}$ such that
$$
\langle\bu_j^{(i_1)},\bu_j^{(i_2)}\rangle =\left\{\begin{array}{ll}1& {\rm if\ } i_1=i_2\\ 0& {\rm otherwise}\end{array}\right..
$$
If $\bA$ is $\mu$-incoherent in that
$$
\|\bu_j^{(i)}\|_{\ell_\infty}\le \sqrt{\mu\over d},\qquad j=1,\ldots, k, i=1,\ldots, r.
$$
then
$$\|\bA\|_{\max}\le \mu^{k/2}rd^{-k/2}.$$
The approximation error bound given by (\ref{eq:appr}) can now be further simplified as
$$
\|\widehat{\bA}-\bA\|\lesssim 
\mu^{k-1}r^{k/2}\max\left\{\left({d\log d\over n}\right)^{1/2},{d^{3/2}\log d\over n}\right\}. 
$$
In other words, when $\mu^{k-1}=O(1)$, we can approximate $\bA$ up to an error of $\epsilon$, in terms of the usual spectral norm, based on observations from
$$
n \ge C_k\max\left({r^kd\log d\over \epsilon^2},{r^{k/2}d^{3/2}\log d\over \epsilon}\right)
$$
entries for some constant $C_k$. If the condition on $\bA$ is strengthened to $\|\bA\|_{\max}\lesssim \mu^{k/2}r^{1/2}d^{-k/2}$, then the sample size requirement becomes 
$$
n \ge C_k\max\left({r^{k-1}d\log d\over \epsilon^2},{r^{(k-1)/2}d^{3/2}\log d\over \epsilon}\right). 
$$
This example shows the importance of leveraging the information that a tensor is incoherent.

\section{Tensor Completion}
\label{sec:main}

We now turn our attention back to tensor completion through incoherent nuclear norm minimization:
\bel{eq:defhbT}
\min_{\bX}\|\bX\|_{\star,\bdelta}\ \hbox{ subject to }\ \calP_\Omega \bX = \calP_\Omega \bT. 
\eel
Denote by $\hbT$ the solution to the above convex optimization problem. We shall utilize the results from the previous sections to establish the requirement on the sample size $n:=|\Omega|$ so that $\hbT=\bT$ with high probability when $\Omega$ is a uniformly sampled subset of $[d_1]\ktimes [d_k]$.

Recall that $r_j(\bT)$s are the Tucker ranks of $\bT$. For brevity, we shall omit the dependence of $r_j$s on $\bT$ for the rest of the section. Denote by
$$r_\ast = \left[{1\over kd}\sum_{j=1}^k \left({d_j\over r_j}\prod_{\ell=1}^kr_\ell\right)\right]^{1/(k-1)},$$
\bel{coherence}
\mu_* = \frac{d_*^k}{kr_\ast^{k-1}d} \max_{i_1,\ldots,i_k}\|\calQ_{\bT}(e_{i_1}\kotimes e_{i_k})\|_{\rm HS}^2,
\eel
and
\bel{spike}
\alpha_\ast = (d_\ast^k/r_\ast)^{1/2}\|\bW_0\|_{\max},
\eel
where as before, $d$ and $d_\ast$ are the arithmetic and geometric averages of $d_j$s, and $\bW_0 \in\R^{d_1\ktimes d_k}$ is the dual of $\bT$ as specified in Theorem \ref{tensor-prop-3}. We are now in position to state our main result.

\begin{theorem}\label{th:main}
Let $\Omega$ be a uniformly sampled subset of $[d_1] \ktimes [d_k]$ and $\hbT$ be the solution to (\ref{eq:defhbT}) with $\delta_j=\sqrt{\lambda_\ast r_\ast/d_j}$. There exists a constant $c_k>0$ depending on $k$ only so that $\P\{\hbT=\bT\}\ge 1-d^{-\beta}$ if
$$
\lambda_\ast\ge {1\over r_\ast}\max_{1\le j\le k} \{\mu_j(\bT)r_j(\bT)\},
$$
and
$$
n:=|\Omega|\ge c_k(1+\beta)\left((\mu_*+\alpha_*^2\lam_*^{k-2})r_*^{k-1} d(\log d)^2 + \alpha_*\lam_*^{k/2-1}r_*^{(k-1)/2}d^{3/2}(\log d)^2 \right)
$$
\end{theorem}

\vskip 15pt
\begin{proof}[Proof of Theorem \ref{th:main}]
The main steps of the proof is analogous to those from Yuan and Zhang (2014). We shall outline below these steps while highlighting the key differences moving from third order tensors to higher order tensors, and from usual tensor nuclear norm to incoherent tensor nuclear norm. We begin with a lemma that reduces the problem to finding a dual certificate.

\begin{lemma}\label{prop-dual-cert} 
Suppose there exists a tensor $\tbG\in\R^{d_1\ktimes d_k}$ such that $\tbG=\calP_\Omega\tbG$, 
\bel{dual-cert-a}
\|\calQ_{\bT}\tbG-\bW_0\|_{\rm HS} < \frac{\sqrt{n/(2d_*^k)}}{k(k-1)}
\eel
and
\bel{dual-cert-b}
\max_{\|\calQ_{\bT}^\perp\bX\|_{\star,\bdelta}=1}\langle\tbG,\calQ_{\bT}^\perp\bX\rangle <\frac{1}{k(k-1)}. 
\eel
If in addition,
\bel{eq:proj-on-T}
\left\|\calP_\Omega|_{{\rm range}(\calQ_{\bT})}\right\|_{\rm HS\to HS}:=\inf\left\{\|\calP_\Omega\calQ_{\bT}\bX\|_{\rm HS}: \|\calQ_{\bT}\bX\|_{\rm HS}=1\right\}\ge \sqrt{n\over 2d_*^k},
\eel
then $\hbT=\bT$.
\end{lemma} 

The proof of Lemma \ref{prop-dual-cert} is relegated to the proof. In the light of Lemma \ref{prop-dual-cert}, it now suffices to verify condition (\ref{eq:proj-on-T}) and construct a dual certificate $\tbG$ that satisfies conditions (\ref{dual-cert-a}) and (\ref{dual-cert-b}). We first verify condition (\ref{eq:proj-on-T}).

Recall that for a linear operator $\calR: \R^{d_1\ktimes d_k}\to \R^{d_1\ktimes d_k}$,
$$
\|\calR\|_{\rm HS\to HS} = \max\left\{\|\calR\bX\|_{\rm HS}: \bX\in\R^{d_1\ktimes d_k}, \|\bX\|_{\rm HS}\le 1\right\}.
$$
Here we prove that under the Hilbert-Schmidt norm in the range of $\calQ_{\bT}$, 
\bel{eq:cond1}
\left\|\calQ_{\bT}\Big((d_*^k/n)\calP_\Omega - \calI\Big)\calQ_{\bT}\right\|_{\rm HS\to HS} \le 1/2
\eel
with large probability. This implies that as an operator in the range of $\calQ_{\bT}$, the spectrum of $(d_*^k/n)\calQ_{\bT}\calP_\Omega\calQ_{\bT}$ is contained in $[1/2,3/2]$. Consequently, (\ref{eq:proj-on-T}) holds via
$$
(d_*^k/n)\|\calP_\Omega\calQ_{\bT}\bX\|_{\rm HS}^2 =  \left\langle \calQ_{\bT}\bX, (d_*^k/n)\calQ_{\bT}\calP_\Omega\calQ_{\bT}\bX\right\rangle 
\ge \frac{1}{2}\|\calQ_{\bT}\bX\|_{\rm HS}^2. 
$$
This goal can be achieved by invoking the following result.

\begin{lemma}
\label{lm:op-norm}
Let $\Omega$ be a uniformly sampled subset from $[d_1]\ktimes [d_k]$ without replacement. Then, 
$$
\P\left\{\left\|\calQ_{\bT}\left({d_*^k\over n}\calP_\Omega - \calI\right)\calQ_{\bT}\right\|_{\rm HS\to HS}\ge \tau\right\}\le 2kr_\ast^{k-1}d\,\exp\left(-\frac{\tau^2/2}{1+2\tau /3}\left(\frac{n}{k\mu_*r_\ast^{k-1}d}\right)\right). 
$$
\end{lemma}

Lemma \ref{lm:op-norm} can be proved using the same argument from Yuan and Zhang (2014) in treating low-rank tensors, noting that
$$
\rank(\calQ_{\bT}) = \dim\big(\hbox{\rm range}(\calQ_{\bT})\big) \le \sum_{j=1}^k d_j\prod_{\ell\neq j}r_\ell = r_\ast^{k-1}d.
$$
The details are omitted for brevity.

Equation (\ref{eq:cond1}) follows immediately from Lemma \ref{lm:op-norm} as soon as
$$
n\ge c_k(\beta+1)\mu_\ast r_\ast^{k-1}d\log(d).
$$

It now remains to show that there exists a dual certificate $\tbG$ that satisfies conditions (\ref{dual-cert-a}) and (\ref{dual-cert-b}). To this end, we apply the now standard ``Golfing scheme''. See, e.g., Gross (2011) and Recht (2011). As argued by Yuan and Zhang (2014), we can construct a sequence $\{\omega_i: 1\le i\le n\}$ of iid uniform vectors from $[d_1]\ktimes[d_k]$ such that $\omega_i\in \Omega$ for all $1\le i\le n$. Let $n_1$ and $n_2$ be two natural numbers to be specified later so that $n_1n_2\le n$. Write
$$
\Omega_j=\left\{\omega_i: (j-1)n_1<i\le jn_1\right\},
$$
for $j=1,2,\ldots, n_2$. Define 
\bel{R_j}
\calR_j = \calI - \frac{1}{n_1}\sum_{i=(j-1)n_1+1}^{jn_1}d_*^k\,\calP_{\omega_i}
\eel
and
\bel{dual-cert-def}
\tbG_j = \sum_{\ell=1}^j \big(\calI - \calR_\ell\big)\calQ_{\bT}\calR_{\ell-1}\calQ_{\bT}\cdots\calQ_{\bT}\calR_{1}\calQ_{\bT}\bW_0,\quad \tbG = \tbG_{n_2}.
\eel
Since $\omega_i\in\Omega$,
$$\calP_\Omega(\calI- \calR_j)=\calI - \calR_j,$$
so that $\calP_\Omega\tbG = \tbG$. It follows from the definition of $\tbG_j$ that 
\bes
\calQ_{\bT}\tbG_j &=& \sum_{\ell=1}^{j}(\calQ_{\bT}- \calQ_{\bT}\calR_\ell\calQ_{\bT})(\calQ_{\bT}\calR_{\ell-1}\calQ_{\bT})\cdots(\calQ_{\bT}\calR_{1}\calQ_{\bT}\bW_0)\\
&=& \bW_0 - (\calQ_{\bT}\calR_{j}\calQ_{\bT})\cdots(\calQ_{\bT}\calR_1\calQ_{\bT})\bW_0
\ees
and for any $\bX\in \R^{d_1\ktimes d_k}$,
$$
\langle\tbG_j,\calQ_{\bT}^\perp\bX\rangle = - \Big\langle \sum_{\ell=1}^j \calR_\ell(\calQ_{\bT}\calR_{\ell-1}\calQ_{\bT})\cdots(\calQ_{\bT}\calR_{1}\calQ_{\bT})\bW_0,\calQ_{\bT}^\perp\bX\Big\rangle. 
$$
Thus, conditions (\ref{dual-cert-a}) and (\ref{dual-cert-b}) hold if 
\bel{dual-cert-1}
\|(\calQ_{\bT}\calR_{n_2})\cdots(\calQ_{\bT}\calR_1)\bW_0\|_{\rm HS} < \frac{\sqrt{n/(2d_*^k)}}{k(k-1)}
\eel
and 
\bel{dual-cert-2}
\left\|\sum_{\ell=1}^{n_2} \calR_\ell(\calQ_{\bT}\calR_{\ell-1}\calQ_{\bT})\cdots(\calQ_{\bT}\calR_{1}\calQ_{\bT})\bW_0\right\|_{\circ,\bdelta} <\frac{1}{k(k-1)}. 
\eel

We still need to prove that (\ref{dual-cert-1}) and (\ref{dual-cert-2}) hold with high probability. For this purpose, we need large deviation bounds for the average of certain iid tensors under the operator, maximum and spectrum norms. The large deviation bounds for the operator and maximum norms are presented in the following lemma. 

\begin{lemma}\label{lm-tensor-ineq} 
Let $\omega_i$, $i=1,\ldots, n_1$ be iid uniformly sampled from $[d_1]\ktimes[d_k]$, and
$$\calD_i = \calQ_{\bT} (d_*^k\calP_{\omega_i})\calQ_{\bT} - \calQ_{\bT}.$$ 
Then, for all $\tau>0$, 
\bel{tensor-ineq-1} 
\P\left\{\left\|\frac{1}{n_1}\sum_{i=1}^{n_1}\calD_i\right\|_{\rm HS\to HS} > \tau \right\}\le 2(r_\ast^{k-1}d)\exp\left(-\frac{\tau^2/2}{1+2\tau /3}\left(\frac{n_1}{\mu_*r_\ast^{k-1}d}\right)\right). 
\eel
Moreover, for any deterministic $\bX\in\R^{d_1\ktimes d_k}$ with $\|\bX\|_{\max}\le 1$, 
\bel{tensor-ineq-2}
\P\left\{\left\|\frac{1}{n_1}\sum_{i=1}^{n_1}\calD_i\bX\right\|_{\max} \ge \tau\right\} \le 2d_*^k\,\exp\left(-\frac{\tau^2/2}{1+2\tau /3}\left(\frac{n_1}{\mu_*r_\ast^{k-1}d}\right)\right). 
\eel
\end{lemma}

Lemma \ref{lm-tensor-ineq} again follows from identical arguments used by Yuan and Zhang (2014) and the details are omitted for brevity. 

Let
$$\bW_j=(\calQ_{\bT}\calR_{j}\calQ_{\bT})\cdots(\calQ_{\bT}\calR_1\calQ_{\bT})\bW$$
with $\bW_0=\bW$. Since $\calR_j$s are iid operators with
$$\calQ_{\bT}\calR_1\calQ_{\bT}= - (1/n_1)\sum_{i=1}^{n_1}\calD_i,$$ 
Equation (\ref{tensor-ineq-1}) yields 
\bes%l{tensor-ineq-1a}
&& \P\left\{\|\bW_j\|_{\rm HS}\le \tau_1^{j}\|\bW\|_{\rm HS}, 1\le j\le n_2\right\}\\
&=& \P\left\{\|(\calQ_{\bT}\calR_j\calQ_{\bT})\cdots(\calQ_{\bT}\calR_1\calQ_{\bT})\bW\|_{\rm HS}\le \tau_1^j\|\bW\|_{\rm HS}, 1\le j\le n_2\right\}\\
&\ge&  1-n_22(r_\ast^{k-1}d)\exp\left(-\frac{\tau_1^2/2}{1+2\tau_1 /3}\left(\frac{n_1}{\mu_*r_\ast^{k-1}d}\right)\right).
\ees%l
This can be used to verify (\ref{dual-cert-1}) with certain $\tau_1$ satisfying 
$$
\tau_1^{n_2}\|\bW\|_{\rm HS}\le \frac{\sqrt{n/(2d_*^k)}}{k(k-1)},
$$
by taking
$$
n\ge n_1n_2\ge c_k(\beta+1)\mu_\ast r_\ast^{k-1}d\log^2(d).
$$

Finally, we prove (\ref{dual-cert-2}). It follows from (\ref{tensor-ineq-2}) that 
\bel{tensor-ineq-2a}\nonumber
&& \P\left\{\|\bW_j\|_{\max} = \|(\calQ_{\bT}\calR_{j}\calQ_{\bT})\cdots(\calQ_{\bT}\calR_1\calQ_{\bT})\bW\|_{\max} \le \tau^{j}\|\bW\|_{\max}, \ 1\le j\le n_2\right\} \\
&\ge&  1-2n_2d_*^k\exp\left(-\frac{\tau^2/2}{1+2\tau /3}\left(\frac{n_1}{\mu_*r_\ast^{k-1}d}\right)\right).
\eel
It follows from the definition of $\calR_j$ in (\ref{R_j}) that for any $\bX$ with $\calQ_{\bT}\bX=\bX$, 
$$
\calR_j\bX =  - \frac{1}{n_1}\sum_{i=(j-1)n_1+1}^{jn_1}\left((d_*^k)\calP_{\omega_i} - \calI \right)\bX. 
$$
Recall that
$$\|\bW\|_{\max}= \alpha_\ast(r_\ast/ d_*^k)^{1/2}.$$ 
Note that $\{\omega_i: (j-1)n_1< i\le jn_1\}$ is independent of $\bW_{j-1}$ and $\calQ_{\bT}\bW_{j-1}=\bW_{j-1}$. By Theorem \ref{con-th-2}, we have
\bes
&&\P\left\{\left\|\calR_j\bW_{j-1}\right\|_{\circ,\bdelta} > \tau^{j-1}t, \|\bW_{j-1}\|_{\max}/\tau^{j-1}\le \|\bW\|_{\max}\right\}\\
&\le& k^2d^{-\alpha}/2+(k^2(\log_2 d)^2/4)
\left\{\exp\left(-4kd\right) + \exp\left(-\sqrt{4kd(3\alpha+7)\log d}\right)\right\}\\
&=:&p_{n_1}(t).
\ees
We note that as $\delta_j=\sqrt{\lambda_\ast r_\ast/d_j}$ and 
$\alpha_\ast = (d_\ast^k/r_\ast)^{1/2}\|\bW_0\|_{\max}$, 
\bes
t &\ge& \frac{c_k^\prime}{n_1}(3\alpha+7)\sqrt{d\log d}(\lam_*r_*)^{k/2}\alpha_*r_*^{1/2}
\max_{1\le j_1<j_2\le k}\left\{(\lam_*r_*)^{-2}(\alpha+1)\left(n_1+d_{j_1}d_{j_2}\log d\right)\right\}^{1/2}
\cr &=& \frac{c_k^\prime}{n_1}\sqrt{d\log d_*}(\delta_*^kd_\ast^k\|\bW\|_{\max})
\max_{1\le j_1<j_2\le k}\left\{\left({n_1\over \delta_{j_1}^2d_{j_1}\delta_{j_2}^2d_{j_2}}
+\frac{\log d}{\delta_{j_1}^2\delta_{j_2}^2}\right)\right\}^{1/2}
\ees
with $c_k^\prime=2^k160$. 
%\bes
%&&\P\left\{\left\|\calR_j\bW_{j-1}\right\|_{\circ,\bdelta} > \tau^{j-1}t, \|\bW_{j-1}\|_{\max}/\tau^{j-1}\le \|\bW\|_{\max}\right\}\\% \le p_{n_1}\left(t\right),
%&\le&kd^{-\alpha}+2k^2\log^k(ed_\ast) \exp\left(k^2d\log(d_\ast)-{n_1t^2\over 2^{2k+4}k^2\alpha_\ast^2r_\ast\log^2(ed_\ast)}\right)\\
%&&+2k^2\log^k(ed_\ast) \exp\left(-{n_1t\over 2^{k+2}k\lambda_\ast^{k/2-1}\alpha_\ast r_\ast^{1/2}d\log(ed_\ast)\|\bA\|_{\max}}\right)\\
%&=:&p_{n_1}(t).
%\ees
%provided that
%$$
%t\ge {1\over n}(\alpha+1)2^{k+3}k^{5/2}\log^2(ed_\ast)\alpha_\ast r_\ast^{1/2}\lambda_\ast^{k/2-1}d^{3/2}\log d.
%$$
Together with (\ref{tensor-ineq-2a}), this yields
\bes
\nonumber
&& \P\left\{\left\|\sum_{j=1}^{n_2} \calR_j(\calQ_{\bT}\calR_{j-1}\calQ_{\bT})\cdots(\calQ_{\bT}\calR_{1}\calQ_{\bT})\bW\right\|_{\circ,\bdelta} < \frac{1}{k(k-1)}\right\}\\
\nonumber
&\ge &  \P\left\{\left\|\calR_j\bW_{j-1}\right\|_{\circ,\bdelta} < \frac{\tau^{j-1}-\tau^j}{k(k-1)}, \|\bW_{j-1}\|_{\max}/\tau^{j-1}\le \|\bW\|_{\max}, \ j\le n_2\right\} \\
&\ge & 1- n_2 p_{n_1}\left(\frac{1-\tau}{k(k-1)}\right) - 2n_2d_*^k \exp\left(-\frac{\tau^2/2}{1+2\tau /3}\left(\frac{n_1}{\mu_*r_\ast^{k-1}d}\right)\right),
\ees
which completes the proof.
\end{proof}
\vskip 15pt

\section{Concluding Remarks}
\label{sec:dis}
We introduce a general framework of nuclear norm minimization for tensor completion and investigate the minimum sample size required to ensure prefect recovery. Our work contributes to a fast-growing literature on higher order tensors, beyond matrices. In particular, we argue that incoherence may play a more prominent role in higher order tensor completion. We show that, by appropriately incorporating information about the incoherence of a $k$th order tensor of rank $r$ and dimension $d\ktimes d$, we can complete it with $O((r^{(k-1)/2}d^{3/2}+r^{k-1}d)(\log(d))^2)$ uniformly sampled entries. This sample size requirement agrees with existing results on recovering a third order tensor (see, e.g., Yuan and Zhang, 2014), and more interestingly, it depends on $k(\ge 3)$ only through the $O(1)$ factor for rank one tensors ($r=1$).

One of the chief challenges when dealing with higher order tensors is computation. Although convex, nuclear norm minimization for higher order tensors is computationally expensive in the worst case. See, e.g., Hillar and Lim (2013). Various relaxations and approximate algorithms have been introduced in recent years to alleviate the computational burden associated with evaluating tensor norms. See, e.g., Nie and Wang (2014), Jiang, Ma and Zhang (2015) and references therein. It is of great interest to study how these techniques can be adopted in the context of tensor completion in general, and nuclear norm minimization in particular. More detailed investigation along this direction is beyond the scope of the current work and we hope to report our progress elsewhere in the near future. Nevertheless, our results here may provide valuable guidance along this direction. For example, our analysis suggests that when developing effective approximation algorithms for higher order tensor completion, it could tremendously beneficial to explicitly take incoherence into account.

\appendix

\section{Proof of Lemma \ref{lm-A}}
It suffices to prove the lemma for $c=1$. 
Consider without loss of generality $\ba$ and $\bu$ with nonnegative components, 
$\|\bu\|_{\ell_2} = 1$ and $\|\bu\|_{\ell_\infty}\le\delta$. Let 
\bes
\bv = (u_1\vee d^{-1/2},\ldots,u_d\vee d^{-1/2})^\top/\sqrt{2},
\ees 
where $a\vee b=\max\{a,b\}$. We have 
$$\|\bv\|_{\ell_\infty} \le \delta/\sqrt{2}, \qquad \sqrt{2}\bv^\top\ba \ge \bu^\top\ba,$$
and
$$\|\bv\|_{\ell_2}^2 = 2^{-1}\sum_{i=1}^d \max(u_i^2,1/d)\le 1.$$
Let 
\bes
\bw\in\{2^{j/2}/\sqrt{2d}, j=0,\ldots,m\}^d\ 
\qquad\hbox{ with }\ w_i \le v_i \le \sqrt{2}w_i,\ \forall\ i=1,\ldots,d. 
\ees
This is possible as
$$\|\bv\|_{\ell_\infty}\le \delta/\sqrt{2}\le \sqrt{2}(2^{m/2}/\sqrt{2d}).$$ 
We have
$$\|\bw\|_{\ell_2}\le \|\bv\|_{\ell_2}\le 1\qquad {\rm and} \qquad 2\bw^\top\ba \ge \sqrt{2}\bv^\top\ba\ge \bu^\top\ba.$$

It remains to count the cardinality. 
Let $\ell_j = \lfloor d/(2^j-1)\rfloor$. For $1\le j\le m$, 
\bes
(2^j/(2d))\left|\{i: w_i^2 = 2^{j}/(2d)\}\right|+ (2d)^{-1}\left[d - \left|\{i: w_i^2 = 2^{j}/(2d)\}\right|\right]\le 1, 
\ees
so that
$$|\{i: w_i^2 = 2^{j}/(2d)\}|\le \ell_j.$$ 
As a choice of $\bw$ can be made by first picking the sign of its elements, 
the cardinality of the $\bw$-collection is no greater than 
\bes
N = 2^{d}\prod_{j=1}^{m}\sum_{0\le \ell \le \ell_j}{d\choose \ell}. 
\ees
Moreover, for $j\ge 2$, we have $\ell_j \le d/(2^j-1)$, so that 
\bes
\sum_{\ell =1}^{\ell_j}{d\choose \ell} 
\le {d\choose \ell_j}\sum_{\ell=0}^{\ell_j}\Big(\frac{1/(2^j-1)}{1-1/(2^j-1)}\Big)^{\ell_j-\ell}
\le {d\choose \ell_j}\left(1+\frac{1}{2^j-3}\right). 
\ees 
It follows with an application of the Stirling formula that 
\bes
N \le 4^d\exp\left\{\sum_{j=2}^{m}\Big(\ell_j\log(ed/\ell_j)+ \frac{1}{2^j-3}\Big)\right\}. 
\ees
Since $x(1+\log(d/x))$ is increasing in $x$ for $0\le x\le d$ and $\ell_j\le d/(2^{j}-1)$,
\bes
\log N \le d\log 4+d \sum_{j=2}^\infty \frac{1+\log(2^j-1)}{2^j-1} +\sum_{j=2}^{\infty}\frac{1}{2^j-3}  
\le 3.082 \times d + 1.344.
\ees
The proof is now completed.

\section{Proof of Lemma \ref{prop-dual-cert}}
Let $\bDelta = \hbT - \bT$. Then $\calP_\Omega \bDelta=0$ and
$$
\|\bT+\bDelta\|_{\star,\bdelta}\le \|\bT\|_{\star,\bdelta}.
$$
It follows from Theorem \ref{tensor-prop-2} that  
$$
\|\bT+\bDelta\|_{\star,\bdelta} \ge \|\bT\|_{\star,\bdelta} + \frac{\|\calQ_{\bT}^\perp\bDelta\|_{\star,\bdelta}}{k(k-1)/2}+\langle \bW_0,\bDelta\rangle. 
$$
Because $\calQ_{\bT}\bW_0=\bW_0$ and
$$\langle\tbG,\bDelta\rangle = \langle\calP_\Omega\tbG,\bDelta\rangle = \langle\tbG,\calP_\Omega\bDelta\rangle=0$$
we get
\bes
- \frac{\|\calQ_{\bT}^\perp\bDelta\|_{\star,\bdelta}}{k(k-1)/2} &\ge& \langle \bW_0-\tbG,\bDelta\rangle \\
&=&\langle \calQ_{\bT}(\bW_0-\tbG), \bDelta\rangle- \langle\tbG, \calQ_{\bT}^\perp\bDelta\rangle\\
&\ge&-\|\bW_0-\calQ_{\bT}\tbG\|_{\rm HS}\|\calQ_{\bT}\bDelta\|_{\rm HS}-\|\calQ_{\bT}^\perp\bDelta\|_{\star,\bdelta}/\{k(k-1)\}.
\ees
It follows that 
$$
\|\calQ_{\bT}^\perp\bDelta\|_{\star,\bdelta}/\{k(k-1)\}\le \|\bW_0-\calQ_{\bT}\tbG\|_{\rm HS}\|\calQ_{\bT}\bDelta\|_{\rm HS}.
$$

Recall that
$$\calP_\Omega\bDelta=\calP_\Omega\calQ_{\bT}^\perp\bDelta+\calP_\Omega\calQ_{\bT}\bDelta=0.$$ 
Thus, in view of (\ref{eq:proj-on-T}) and Proposition \ref{tensor-prop-1}
\bel{eq:hsbd}
 \frac{\|\calQ_{\bT}\bDelta\|_{\rm HS}}{\sqrt{2d_*^k/n}}
 \le \|\calP_\Omega\calQ_{\bT}\bDelta\|_{\rm HS}=\|\calP_\Omega\calQ_{\bT}^\perp\bDelta\|_{\rm HS}
 \le \|\calQ_{\bT}^\perp\bDelta\|_{\rm HS}
 \le \|\calQ_{\bT}^\perp\bDelta\|_{\star,\bdelta}. 
\eel
Consequently, 
$$
\frac{\|\calQ_{\bT}^\perp\bDelta\|_{\star,\bdelta}}{k(k-1)}\le \sqrt{2d_*^k/n}\|\bW_0-\calQ_{\bT}\tbG\|_{\rm HS}\|\calQ_{\bT}^\perp\bDelta\|_{\star,\bdelta}.
$$
Because of (\ref{dual-cert-a}), we have $\|\calQ_{\bT}^\perp\bDelta\|_{\star,\bdelta}=0$. Together with (\ref{eq:hsbd}), we conclude that $\bDelta=0$, or equivalently $\hbT=\bT$.
\end{document}